\documentclass[12pt]{iopart}

\usepackage{iopams}  

\usepackage{appendix}

\usepackage{amsthm}
\usepackage{amssymb}

\usepackage{color} 

\theoremstyle{definition}
\newtheorem{theorem}{Theorem}[section]
\newtheorem{definition}[theorem]{Definition}

\newtheorem{proposition}[theorem]{Proposition}

\newtheorem{remark}[theorem]{Remark}

\newcommand{\forward}{\mathcal{F}}
\newcommand{\inverse}{\mathcal{I}}
\newcommand{\K}{\mathcal{K}}

\begin{document}

%\title[]{Inverse Born Again}
\title[Analysis of the inverse Born series]{Analysis of the inverse Born series: an approach through geometric function theory}

\author{Jeremy G Hoskins}
\address{Department of Statistics, University of Chicago, Chicago, IL, USA}
\ead{jeremyhoskins@uchicago.edu}

\author[Schotland]{John C Schotland}
\address{Department of Mathematics and Department of Physics, Yale University, New Haven, CT, USA}
\ead{john.schotland@yale.edu}

\begin{abstract}
We analyze the convergence and approximation error of the inverse Born series, obtaining results that hold under qualitatively weaker conditions than previously known. Our approach makes use of tools from geometric function theory in Banach spaces. An application to the inverse scattering problem with diffuse waves is described.
\end{abstract}

\vspace{2pc}

%\submitto{\IP}
%\maketitle

\section{Introduction}

The inverse Born series (IBS) is a reconstruction method that is applicable to a wide variety of inverse problems. It was initially introduced in the setting of the one-dimensional inverse backscattering problem for the Schrodinger equation~\cite{jost_kohn,moses}. Later, it was extended to higher dimensions and to a variety of scattering experiments in both classical and quantum physics~\cite{prosser,devaney_wolf,louis,weglein_03}. In more recent work, the IBS has been applied to the inverse problems of optical tomography, electrical impedance tomography, and acoustic and electromagnetic imaging~\cite{moskow_1,markel,moskow_2,panasyuk,kilgore_2017,arridge,kilgore,machida,bardsley}. Moreover, the convergence, stability and error of the method have been analyzed. The analysis exploits the combinatorial structure of the IBS in combination with partial differential equation (PDE) estimates~\cite{moskow_1}. See~\cite{review} for a survey of these developments.

In previous work, we have applied the IBS to inverse problems on graphs. The motivation was to isolate the combinatorial structure in a simpler discrete setting~\cite{chung}. To this end, we have shown that the IBS converges under qualitatively weaker conditions than those that hold in the continuum. This result was obtained by making use of tools from geometric function theory in several complex variables, especially estimates of Bloch radii. The key idea is to bound the size of a polydisc on which the inverse of a holomorphic function exists.

In this work we extend the approach of~\cite{chung} to PDE inverse problems. As explained below, we obtain weaker conditions for the convergence and approximation error of the series compared to the findings of~\cite{moskow_1}. Our results obtain from an alternative, but equivalent, formulation of the IBS within which it is possible to construct the analog of Bloch radii in Banach spaces.

The remainder of this paper is organized as follows. In section 2 we formulate the Born and inverse Born series in Banach spaces. We then state our main theorems on the convergence and approximation error of the IBS. The proofs of these results are presented in section 3. In section 4 we apply our results to the inverse problem of optical tomography with diffuse light. The Appendix discusses an alternative formulation of the IBS and compares it to the classical case described in~\cite{review}.

We use the following notational convention throughout this paper. If $X$ is a Banach space, $X^n$ indicates the $n$-fold tensor product $X^n = X\otimes\cdots\otimes X$ equipped with the projective norm for $n> 1$. We note that $X^n$ is generally not a Banach space.

\section{Main results}
Let $X$ and $Y$ be Banach spaces and $K_m:X^m\rightarrow Y$  be $m$-multilinear operators with $m\ge 1$. Consider the operator $\forward:X \rightarrow Y$ defined by
\begin{eqnarray}
\label{eq:forward_brn}
\forward[\eta] = \sum_{m=1}^\infty K_m(\eta,\dots,\eta).
\end{eqnarray}
We will refer to the $K_m$  as forward operators. The forward problem is to evaluate the map $\forward: \eta\mapsto \phi$ for $\eta \in X$ and $\phi\in Y$.

The inverse problem is to determine $\eta$ assuming $\phi$ is known. That is, we wish to construct a map $\inverse: Y \rightarrow X$ which is, in a suitable sense, the inverse of $\forward$. To proceed, we define the operator $\inverse$  by
\begin{eqnarray}
\label{inverse_born}
\inverse[\phi] = \sum_{m=1}^\infty \mathcal{K}_m(\phi).
\end{eqnarray}
The inverse operator $\mathcal{K}_m: Y \rightarrow X$ is homogeneous of degree $m$ and is given by
\begin{eqnarray}
\label{inv_operators}
\mathcal{K}_1 (\phi) &= K_1^{+} (\phi),\\
\mathcal{K}_2(\phi) &=-\mathcal{K}_1\left(K_2 (\mathcal{K}_1(\phi),\mathcal{K}_1(\phi))\right),\\
\mathcal{K}_m(\phi) &= -\sum_{n=2}^{m}\sum_{i_1+\cdots+i_n = m}  \K_1{K}_n \left( \mathcal{K}_{i_1}(\phi), \dots, \mathcal{K}_{i_n}(\phi) \right) .
\end{eqnarray}
Here $K_1^+$ denotes the inverse of $K_1$, provided $K_1^{-1}$ is bounded, or a regularized pseudoinverse of $K_1$. 

The series (\ref{eq:forward_brn}) and (\ref{inverse_born}) will be referred to as the forward Born series and inverse Born series, respectively. 

\begin{remark}
We note that the above formulation of the IBS differs, to some extent, from the series analyzed in \cite{moskow_1}. See Appendix A for a discussion of this point. 
\end{remark}

Throughout this paper we denote by $B_{R,X}$ the ball of radius $R$ centered at the origin in the Banach space $X$. We omit the $X$ when there is no ambiguity as to the Banach space in question. The following theorem establishes a sufficient condition for convergence of the IBS.

\begin{theorem}[Convergence of the inverse Born series]
\label{thm:conv_inv}
Let $\mu$ and $\nu$ be positive constants. Suppose that  $\| K_m(\eta_1,\dots,\eta_m)\|_{Y} \le \nu \mu^{m-1} \|\eta_1\|_{X}\dots\|\eta_m\|_{X},$ for $m=1,2,\cdots.$ The inverse Born series converges if
$\|\mathcal{K}_1 \phi\|_{X} < r$ ,
where the radius of convergence $r$ is given by
$$
r=\frac{1}{2\mu} \left[\sqrt{16 C^2+1}-4 C \right],
$$
where $C = \max\{2,\|\mathcal{K}_1\|\nu\}.$ Moreover, if $\mathcal{K}_1\phi\in B_r$ then the inverse operator $\inverse$ maps $B_r$ into $B_{r_0}$, with $r_0 = 2\mu/\sqrt{16 C^2+1}$.
\end{theorem}

\begin{remark}
The convergence of the IBS was analyzed in \cite{moskow_1}. It was found that certain smallness conditions on both $\|\K_1\phi\|$ and $\|\K_1\|$ are sufficient to guarantee convergence. In particular, the condition on $\|\K_1\|$ is rarely met in practice. Note that a condition on $\|\K_1\|$ is not present in Theorem~\ref{thm:conv_inv}. 
\end{remark}

The limit of the IBS does not, in general, coincide with $\eta$. We characterize the approximation error as follows.

\begin{theorem}[Approximation error]
\label{thm:error}
Suppose that the hypotheses of Theorem~\ref{thm:conv_inv} hold and that the forward and inverse Born series converge. Let $\tilde\eta$ denote the sum of the inverse Born series and $\eta_1 = \mathcal{K}_1\phi$.  Setting $\mathcal{M} = \max\left\{\|\eta\|_X,\left\|\tilde\eta\right\|_X\right\},$ we further assume that
\begin{eqnarray}
\mathcal{M}  < \frac{1}{\mu}\left(1-\sqrt{\frac{\nu \|\mathcal{K}_1\|}{1+\nu \|\mathcal{K}_1\|}}\right),
\end{eqnarray}
Then the approximation error can be estimated as follows:
\begin{eqnarray}
\label{eq:combined_error}
\nonumber
\left\| \eta - \sum_{m=1}^N \mathcal{K}_m(\phi) \right\|_X  &\le M\left( \frac{\|{\eta_1}\|_X}{r}\right)^{N+1} \frac{1}{1- \frac{\|{\eta_1}\|_X}{r}} \\ 
&+ \left( 1 - \frac{\nu\|\mathcal{K}_1\|}{(1-\mu\mathcal{M})^2}+\nu\|\mathcal{K}_1\|\right)^{-1}\left\| (I-\mathcal{K}_1K_1)\eta\right\|_X,
\end{eqnarray}
where 
\begin{equation*}
\label{eq:M_eq}
 M =\frac{2\mu}{\sqrt{16 C^2+1}}.
 % \frac{1}{\mu}\left( 1- \sqrt{\frac{\nu\|\mathcal{K}_1\|}{1+\nu \|\mathcal{K}_1\|}}\right)
\end{equation*}
\end{theorem}

We note that the second term on the right-hand side of (\ref{eq:combined_error}) corresponds to the absolute error of reconstructing $\eta$.

\section{Proofs of the main results}
Before proceeding, we introduce the notion of {Bloch radii}.
\begin{definition}
%[\cite{fixed-point}] 
Let $B$ be the open unit ball of a complex Banach space and let $h:B\to X$ be a holomorphic function with $h(0) = 0$ and $Dh(0) = Id.$ The positive numbers $r$ and $P$ with $r<1$ are Bloch radii for $h$, if $h$ maps a subdomain of $B_r(0)$ biholomorphically onto $B_P(0).$
\end{definition}

The next theorem provides a bound on the Bloch radii of a function.
\begin{theorem}
[\cite{harris}] 
Let $B$ be the open unit ball of a complex Banach space $X$ and let $h:B\to X$ be a holomorphic function with $h(0) = 0$ and $Dh(0) = Id.$ Suppose that $\|h(x)\|\le M$ for all $x \in B.$ Then $r = 1/\sqrt{4M^2+1}$ and $P=1/(2M+\sqrt{4M^2+1})$ are Bloch radii for $h.$
\end{theorem}

The proof of Theorem \ref{thm:conv_inv} is an immediate consequence of the above theorem.
\begin{proof}
We begin by assuming that $\mu <1$ and define the operator $h:X\to X$ by
\begin{equation}
h[x] = x + \sum_{m=2}^\infty \mathcal{K}_1 K_m(x,\dots,x).
\end{equation}
When $\|x\| \le 1$, it follows that $\|h[x]\| \le 1 + \| \mathcal{K}_1\|\nu\mu \frac{1}{1-\mu}.$ Let $C_0 = 2\max\{1,\|\mathcal{K}_1\|\nu\mu\}.$ Then $\|h[x]\| \le C_0/(1-\mu)$ for all $x \in B.$ From the previous theorem it follows that $h$ has Bloch radii $r = 1/\sqrt{4M^2+1}$ and $P=1/(2M+\sqrt{4M^2+1})$ with $M = C_0/(1-\mu).$ Hence, there exist homogeneous polynomials $\mathcal{A}_m : X \to X$ such that
\begin{equation}
x = h[x] + \sum_{m=2}^\infty \mathcal{A}_m[h[x]],
\end{equation}
whenever $\|x \| < r$ and $h[x] <M$. Moreover, the $\mathcal{A}_m$ are unique. 
%The radius of convergence of the inverse series guarantees that $\|F_m\| < \alpha P^{-m}.$ 
Inserting the definition of $h[x]$ into the above series and comparing terms at the same order in $x$, we see that 
\begin{equation}
\mathcal{K}_m(y) = (\mathcal{A}_m \circ \mathcal{K}_1)(y).
\end{equation}

Now, we consider the case for arbitrary $\mu>0,$ which follows by a straightforward rescaling. For some fixed $\beta > \mu,$ we define 
\begin{equation}
h_\beta[x] = x + \beta^{-(m-1)} \sum_{m=2}^\infty K_m(x,\dots,x).
\end{equation}
Bounding this function in a similar way as was done for $h$, we find that $$\|h_\beta[x]\| \le C/(1-\mu/\beta),$$
 where $C=2 \max\{1, \|\mathcal{K}_1\| \nu \mu/\beta\}.$ Hence, there exist unique homogeneous polynomials $\mathcal{A}^{(\beta)}_m$ such that
\begin{equation}
x = h_\beta[x] + \sum_{m=2}^\infty \mathcal{A}_m^{(\beta)}(h_\beta[x]),
\end{equation}
where the above series converges for $\|h_\beta\| \le 1/(2M+\sqrt{4M^2+1})$ with $M = C/(1-\mu/\beta),$ where  
%Moreover, $\|F^{(\beta)} \| < \alpha_\beta P^{-m}.$ 
Next, we observe that $h_\beta [x] = \beta h [x/\beta].$ It follows that $h^{-1}[y] =  h^{-1}_\beta[\beta y]/\beta.$ In particular, the series of $h^{-1}$ has a radius of convergence bounded from below by $P/\beta.$ Choosing $\beta = 2 \mu,$ we obtain the bound
\begin{equation}
P = \frac{1}{2\mu}\left[\sqrt{16 C^2+1} -4C\right].
\end{equation}
Finally, we note that by construction, any function in $B_P$ is mapped into a function in $B_r,$ $r = 2\mu/\sqrt{16 C^2+1}$.
\end{proof}

We now prove the error estimate for the IBS.

\begin{proof}[Proof of Theorem \ref{thm:error}]
We proceed in two steps. First we give an estimate for the tail of the IBS. Then we derive a bound for the difference between the sum of the inverse Born series and the true $\eta$.

Put $\tilde{\eta} \in B(r,X)$ and $\ell \in B(1,X^*)$, where $X^*$ is the dual space of $X.$ Consider the function
\begin{equation}
f(t) = \ell\left(\sum_{n=1}^\infty \mathcal{K}_n(\tilde{\eta}\,t)\right).
\end{equation}
Note that for all $|t| < r/\|\tilde{\eta}\|_X,$ $f(t)$ is analytic in $t$ and $|f(t)| < M,$ where
\begin{equation}
M = \frac{2\mu}{\sqrt{16 C^2+1}}.
\end{equation}
Using the previous theorem, we note that the ball of radius $r$ is mapped into the ball of radius $r_0={2\mu}/{\sqrt{16 C^2+1}}.$
We further note that since $f(t)$ is analytic for $|t|<r/\|\tilde{\eta}\|_X$. there exist coefficients $c_n(\ell,X,\tilde{\eta})$ such that
\begin{equation}
f(t) = \sum_{n=1}^\infty c_n(\ell,X,\tilde{\eta}) \,t^n.
\end{equation}
Applying Cauchy's estimate, we see that
\begin{eqnarray}
|c_n(\ell,X,\tilde{\eta})| \le M \frac{\|\tilde{\eta}\|_X^n}{r^n},
\end{eqnarray}
from which it follows that
\begin{eqnarray}
\left|\ell\left(\mathcal{K}_n(\tilde{\eta}) \right)\right| \le M \frac{\|\tilde{\eta}\|_X^n}{r^n}.
\end{eqnarray}
Therefore if $\eta^{(K)}$ is the sum of the first $K$ terms of the IBS and $\eta^{(\infty)}$ is the sum of all the terms, 
\begin{eqnarray}
\left|\ell\left(\eta^{(\infty)} - \eta^{(K)} \right)\right| &= \left|\sum_{n>K} \ell\, \mathcal{K}_n(\tilde{\eta}) \right|\\
&\le M \sum_{n>K} \frac{\|\tilde{\eta}\|_X^n}{r^n}\\
&= M \left(\frac{ \|\tilde{\eta}\|_X }{r}\right)^{K+1} \frac{1}{1-\frac{\|\tilde{\eta}\|_x}{r}}.
\end{eqnarray}
Taking the supremum over all $\ell \in B(1,X^*)$ we obtain
\begin{eqnarray}\label{first_error}
\left\| \eta^{(\infty)} - \eta^{(K)}\right\|_X \le M \left(\frac{ \|\tilde{\eta}\|_X }{r}\right)^{K+1} \frac{1}{1-\frac{\|\tilde{\eta}\|_x}{r}}.
\end{eqnarray}

Next we bound the difference between $\eta^{(\infty)}$ and the true $\eta$. We assume that
\begin{equation}
\|\eta\|_X,\,\left\| \eta^{(\infty)}\right\|_X \le \min \left\{\frac{1}{\mu}\left(1-\sqrt{\frac{\nu \|\mathcal{K}_1\|}{1+\nu \|\mathcal{K}_1\|}}\right),r_0\right\}.
\end{equation}
By construction we have that
\begin{eqnarray}
\eta^{(\infty)}+\sum_{m=2}^\infty K_m\left( \eta^{(\infty)},\dots,\eta^{(\infty)}\right)-\mathcal{K}_1\phi = 0.
\end{eqnarray}
Since the true  $\eta$ satisfies $\|\eta\|_X <1/\mu,$  we find that
\begin{equation}
\phi = \sum_{m=1}^\infty K_m(\eta,\dots,\eta).
\end{equation}
Upon substituting this identity in the previous equation, we obtain
\begin{eqnarray}
\eta^{(\infty)} -\eta+\sum_{m=2}^\infty \mathcal{K}_1 \left[K_m\left(\eta^{(\infty)}\right)-K_m(\eta)\right]= -(I-\mathcal{K}_1K_1)\eta.
\end{eqnarray}
Letting $\mathcal{M} = \sup\{ \|\eta\|_X,\|\eta^{(\infty)}\|_X\},$ we bound the left-hand side from below, noting that
\begin{eqnarray}
&\left\|\eta^{(\infty)} -\eta+\sum_{m=2}^\infty \mathcal{K}_1 \left[K_m\left(\eta^{(\infty)}\right)-K_m(\eta)\right]\right\|_X 
\\ & \ge  \|\eta^{(\infty)} -\eta\|_X- \left\|\sum_{m=2}^\infty \mathcal{K}_1 \left[K_m\left(\eta^{(\infty)}\right)-K_m(\eta)\right]\right\|_X\\
&\ge \|\eta^{(\infty)} -\eta\|_X- \nu\|\mathcal{K}_1\| \sum_{m=2}^\infty m  \mu^{m-1}\mathcal{M}^{m-1}\left\|\eta-\eta^{(\infty)}\right\|_X\\
& \ge \left(1-\frac{ \nu \|\mathcal{K}_1\|}{(1-\mu\mathcal{M})^2} +\nu \| \mathcal{K}_1\|\right)\, \|\eta^{(\infty)} -\eta\|_X.
\end{eqnarray}
Thus
\begin{eqnarray}\label{second_error}
 \|\eta^{(\infty)} -\eta\|_X \le \left(1-\frac{\nu\|\mathcal{K}_1\|}{(1-\mu\mathcal{M})^2}+\nu \|\mathcal{K}_1\| \right)^{-1} \left\| (I- \mathcal{K}_1K_1)\eta\right\|_X,
\end{eqnarray}
provided that
\begin{eqnarray}
\mathcal{M}  < \frac{1}{\mu}\left(1-\sqrt{\frac{\nu \|\mathcal{K}_1\|}{1+\nu \|\mathcal{K}_1\|}}\right),
\end{eqnarray}
which holds by assumption. The estimate (\ref{eq:combined_error}) follows immediately from (\ref{first_error}) and (\ref{second_error}), together with the triangle inequality.
\end{proof}
\section{Application to diffuse waves}
We now apply Theorem~\ref{thm:conv_inv} to the inverse scattering problem for diffuse waves. We follow~\cite{moskow_1} with minor modifications. The energy density $u$ of a diffuse wave satisfies 
\begin{eqnarray}
\label{eq:diff_wave}
-\nabla^2 u +k^2(1+\eta(x))u = g , \quad x \in \mathbb{R}^3, \\
\lim_{|x|\to \infty} u = 0.
\end{eqnarray}
where $1+\eta(x)$ is nonnegative, $g$ is the source, and $k$ is a nonnegative constant. Here we assume that the support of $\eta$ lies in $B_a$, the ball of radius $a$, and the sources and detectors lie on a sphere of radius $R>r$, which we denote $\partial B_R.$ In particular, $g$ is a function supported on  $B_R$.

 The solution to (\ref{eq:diff_wave}) can be expressed in the form
\begin{eqnarray}
\label{born_series}
\phi(x,y) = \sum_{m=1}^\infty K_m(\eta,\dots,\eta)(x,y) , \quad x,y \in \partial B_R.
\end{eqnarray}
where
\begin{eqnarray}
\nonumber
K_m(\eta_1,\dots,\eta_m)(x,y) &=-(-k^2)^m \int_{B_a}\dots \int_{B_a} G(x,z_1)G(z_1,z_2)\dots G(z_{m-1},z_m)\\
&\times G(z_m,y) \eta(z_1)\dots\eta(z_m)\,{\rm d}z_1\dots{\rm d}z_m .
\end{eqnarray}
Here $\phi(x,y)$ is the scattered field for a source at the point $x$ and a detector at $y$. The Green's function $G$ is given by
%obeys
%\begin{eqnarray}
%\label{eq:hom_eq}
%-\nabla^2_x G(x,y)+k^2 G(x,y) = \delta(x-y) , \quad x\in \mathbb{R}^3, \\
%\lim_{|x| \to \infty} G(x,y) =0.
%\end{eqnarray}
%In particular, 
\begin{equation}
G(x,y) = \frac{e^{-k|x-y|}}{4\pi|x-y|}.
\end{equation}
Eq.~(\ref{born_series}) is the Born series for diffuse waves. 

The inverse problem is to recover $\eta$ from measurements of $\phi$. The associated IBS is of the form
\begin{eqnarray}
\label{inv_born_series}
\eta(x) = \sum_{m=1}^\infty \mathcal{K}_m(\phi)(x),
\end{eqnarray}
where $\mathcal{K}_m$ is defined by (\ref{inv_operators}). The following proposition provides sufficient conditions for the convergence of the inverse series.

\begin{proposition}
Let $\nu$ and $\mu$ be the positive constants defined by
$$\nu =  k^2 |B_a|^{1/2} \frac{R}{16 \pi a} \log\left| \frac{R^2+a^2}{R^2-a^2}\right|$$
 and 
$$\mu = k^2 \sqrt{\frac{a}{4\pi}},$$
and let $\mathcal{K}_1$ be any bounded linear operator from $L^2(\partial B_R \times \partial B_R) \to L^2(B_a).$ 
The corresponding inverse Born series for diffuse waves converges for all $\phi$ such that
\begin{equation}
\|\mathcal{K}_1 \phi\|_{L^2(B_a)} \le \frac{1}{2{\mu}} \left[\sqrt{16 C^2+1}-4 C \right],
\end{equation}
where $C = \max\{1, \|\mathcal{K}_1\| \nu\}.$
\end{proposition}
\begin{proof}
We begin by observing that 
$$\|K_1(\eta)\|_{L^2(\partial B_R \times \partial B_r)} \le k^2 |B_a|^{1/2} \|\eta\|_{L^2(B_a)}  \sup_{y \in B_a} \|G(y,\cdot)\|_{L^2(\partial B_R)}^2.$$
A straightforward calculation shows that
$$\sup_{y \in B_a} \|G(y,\cdot)\|_{L^2(\partial B_R)}^2 \le \frac{R}{16 \pi a} \log\left| \frac{R^2+a^2}{R^2-a^2}\right|.$$
Similarly, repeated application of Holder's inequality yields
\begin{eqnarray}
&\|K_m(\eta_1,\dots,\eta_m)\|_{L^2(\partial B_R \times \partial B_r)} \le k^{2m} \|\eta_1\|_{L^2(B_a)} \dots  \|\eta_m\|_{L^2(B_a)}  \\
& \quad\quad\times  |B_a|^{1/2} \sup_{y \in B_a} \|G(y,\cdot)\|_{L^2(\partial B_R)}^2 \sup_{y \in B_a} \|G(y,\cdot)\|_{B_a}^{m-1}.
\end{eqnarray}
The last term on the right-hand side of the above expression is easily seen to be bounded by $(a/4\pi)^{(m-1)/2}.$ It follows that if we define 
$$\nu =  k^2 |B_a|^{1/2} \frac{R}{16 \pi a} \log\left| \frac{R^2+a^2}{R^2-a^2}\right|$$
 and 
$$\mu = k^2 \sqrt{\frac{a}{4\pi}},$$
then 
\begin{equation}
\|K_m(\eta_1,\dots,\eta_m)\|_{L^2(\partial B_R \times \partial B_R)} \le \|\eta_1\|_{L^2(B_a)} \dots \|\eta_m\|_{L^2(B_a)} \nu \mu^{m-1}.
\end{equation}
Thus, the Born series
\begin{eqnarray}\label{eq:forward_Born}
\phi(x,y) = \sum_{m=1}^\infty K_m(\eta,\dots,\eta)(x,y),
\end{eqnarray}
converges in $L^2(\partial B_R \times \partial B_R)$ if $\|\eta\|_2 \le 1/{\mu}.$
Moreover, by Theorem~\ref{thm:conv_inv}, for any bounded linear operator $\mathcal{K}_1: L^2(\partial B_R \times \partial B_R)\rightarrow L^2(B_{R_a}),$ the IBS
\begin{equation}
\sum_{m=1}^\infty \mathcal{K}_m(\phi),
\end{equation}
converges for all $\phi$ such that
\begin{equation}
\|\mathcal{K}_1 \phi\|_{L^2(B_a)} \le \frac{1}{2{\mu}} \left[\sqrt{16 C^2+1}-4 C \right],
\end{equation}
where $C = \max\{1, \|\mathcal{K}_1\| \nu\}.$
\end{proof}

\begin{remark}
The above analysis applies more generally, and similar results can be obtained for many other problems where the inverse Born series has been applied~\cite{review}. 
\end{remark}

\section*{Acknowledgements}
We thank Francis Chung, Anna Gilbert and Larry Harris for valuable discussions. This work was supported in part by the NSF grant DMS-1912821 and the AFOSR grant FA9550-19-1-0320.

\appendix

\section{Construction of the inverse Born series}
The inverse Born series analyzed in this paper is slightly different than the classical one described in \cite{moskow_1}. Though the two are equivalent when the linearized operator $K_1$ has a bounded inverse, writing the series as in (\ref{inv_operators}) presents several analytical and computational advantages. Here we describe the construction of both forms of the IBS, and discuss the computational ramifications.

The derivation of both forms of the IBS series starts with the observation that for {$\eta$ small enough},
$$
\phi = K_1(\eta) + K_2(\eta,\eta) + K_3(\eta,\eta,\eta) + \cdots,
$$
for multilinear operators $K_1,$ $K_2,$ $K_3,$ $\ldots.$ Next, we assume that the converse holds,  namely that for $\phi$ sufficiently small, we can expand $\eta$  in of the form
$$\eta = \mathcal{K}_1(\phi) + \mathcal{K}_2(\phi,\phi) + \mathcal{K}_3(\phi,\phi,\phi) + \cdots,$$
for some multilinear operators $\mathcal{K}_1,$ $\mathcal{K}_2,$ $\ldots.$
Substituting the first series into the second, we find that
\begin{eqnarray}
\label{eqn:eta_recurse}
\eta = \K_1( K_1(\eta) + K_2(\eta,\eta) +\cdots) + \K_2(K_1(\eta) +\cdots,K_1(\eta)+\cdots) + \cdots.
\end{eqnarray}
Equating terms of the same order in $\eta$ on the left and right hand sides of the above yields
\begin{eqnarray}
\K_1K_1 &= Id, \\
\K_2 &= -K_2 \circ (\K_1 \otimes \K_1),\\
\K_3 &= -\K_2 \circ (K_1 \otimes K_2) \circ (\K_1 \otimes \K_1 \otimes \K_1) \\
\nonumber
& -\K_2 \circ (K_2 \otimes K_1) \circ (\K_1 \otimes \K_1 \otimes \K_1) - \K_1 \circ K_3 \circ (\K_1 \otimes \K_1 \otimes \K_1),\\
\K_n &= -\sum_{m=1}^{n-1}\sum_{i_1+\dots + i_m = n} \K_m \circ (K_{i_1} \otimes K_{i_m}) \circ \K_1^{n}.
\end{eqnarray}
This is the standard formulation of the IBS presented in \cite{moskow_1}. We remark that only the diagonal elements (at equal arguments) of $\K_2,\K_3,\dots$ are uniquely determined by (\ref{eqn:eta_recurse}), meaning $\K_2(\phi,\phi),$ $\K_3(\phi,\phi,\phi),  \K_m(\phi,\dots,\phi), \cdots.$ In particular, to each term it is possible to add any operator which is antisymmetric in any two of its arguments.

Now, if instead one substitutes the series for $\eta$ into the series for $\phi,$ we find that
$$\phi = K_1(\mathcal{K}_1(\phi) + \mathcal{K}_2(\phi) +\cdots) + K_2(\mathcal{K}_1(\phi) + \cdots,\mathcal{K}_1(\phi)+\cdots) + \cdots.$$
Equating terms of the same order in $\phi$ on the left and right hand sides, we see that
\begin{eqnarray}
\label{syst}
K_1\K_1  &= Id, \\
\K_2(\phi) &= -\K_1 K_2(\K_1(\phi),\K_1(\phi)),\\
\nonumber
\K_3(\phi) &= -\K_1 K_3(\K_1(\phi),\K_1(\phi),\K_1(\phi)) - K_2(\K_1(\phi),\K_2(\phi)) - K_2(\K_2(\phi),\K_1(\phi)),\\ \\
\K_m(\phi) &= -\sum_{n=2}^m \sum_{i_1+\cdots + i_n = m} \K_1 K_n (\K_{i_1}(\phi),\dots, \K_{i_n}(\phi)).\label{syst_last}
\end{eqnarray}
We note that the above recurrence relations only provide the values of $\K_2,\K_3,\dots$ on their diagonals (when all the arguments are equal). On the other hand, unlike the classical IBS, the only terms involving the inverse operators $\K_1,\K_2,\K_3,\dots$ appearing on the right-hand side of (\ref{syst}-\ref{syst_last}) are only evaluated on their diagonals. Thus, unlike the classical IBS, we do not need to evaluate $\K_2,\dots,\K_3,\cdots$ off the diagonal.

Computationally, the above reformulation of the IBS presents a significant advantage. Inspecting the formulae, each $\K_m,$ $m \ge 2,$ is only evaluated at one argument, $\phi.$ Thus, the collection of functions $\K_1(\phi),\K_2(\phi),\dots$ can be stored and re-used, obviating the need for recursive re-evaluations of $\K_2,\K_3,\dots$ at each order. This is in contrast to the classical case, where the arguments of $\K_2$, for example, change depending on which term in the series is being evaluated. 
We intend to explore these observations in future work.

\end{document}